\documentclass[final,1p,number,fleqn,12pt]{elsarticle}
\usepackage{amsmath,amssymb,amsthm,mathptmx}

\newtheorem{theorem}{Theorem}[section]
\newtheorem{definition}[theorem]{Definition}
\newtheorem{lemma}[theorem]{Lemma}
\newtheorem{corollary}[theorem]{Corollary}
\newtheorem{remark}{Remark}
\newtheorem{example}[theorem]{Example}
\DeclareMathOperator{\RE}{Re}
\newcommand{\UD}{\mathbb{D}}

\begin{document}
\begin{frontmatter}

\title{A first-order differential double subordination with applications\tnoteref{t1}}

\author[rma]{Rosihan M. Ali}
\address[rma]{School of Mathematical Sciences,
Universiti Sains Malaysia, 11800 USM, Penang, Malaysia}
\ead{rosihan@cs.usm.my}

\author[Cho]{Nak Eun Cho}
\address[Cho]{Department of Applied Mathematics,
Pukyong National University, Busan 608-737, South Korea}
\ead{necho@pknu.ac.kr}

\author[Kwon]{Oh Sang Kwon }
\address[Kwon]{Department of Mathematics,
Kyungsung University, Busan 608-736, South Korea}
\ead{oskwon@ks.ac.kr}

\author[rma,vravi]{V. Ravichandran}
\address[vravi]{Department of Mathematics,
University of Delhi, Delhi-110 007, India}
 \ead{vravi68@gmail.com}

\tnotetext[t1]{The works by the first and the fourth authors were
supported in part by grants from Universiti Sains Malaysia and
University of Delhi respectively. The second author was supported by
the Basic Science Research Program through the National Research
Foundation of Korea (NRF) funded by the Ministry of Education,
Science and Technology (No. 2009-0066192). This work was completed
during the visit of the first and the fourth authors to Pukyong
National University. The authors are thankful to the referees for
their suggestions that helped improve the presentation of this
manuscript. }

\begin{abstract}  Let $q_1$ and $q_2$ belong to a certain class of
normalized analytic univalent functions in the open unit disk of the
complex plane. Sufficient conditions are obtained for normalized
analytic functions $p$ to satisfy the double subordination chain
$q_1(z)\prec p(z)\prec  q_2(z)$.  The differential sandwich-type
result obtained is applied to normalized univalent functions and to
$\Phi$-like functions.
\end{abstract}


\end{frontmatter}

\section{Introduction}

Let $\mathcal{H}$ be the class consisting of analytic functions  in
the open unit disk $\UD :=\{z\in \mathbb{C}:|z|<1\}$ of the complex
plane $\mathbb{C} $. For $a\in\mathbb{C}$, let $\mathcal{
H}[a,n]:=\{f\in\mathcal{H}: f(z)=a+a_nz^n+a_{n+1}z^{n+1}+\cdots\}$,
and  $\mathcal{A}:=\{f \in \mathcal{H}: f(0)=0, f'(0)=1\}$.  A
function $f \in\mathcal{H}$ is said to be subordinate to an analytic
function $g\in\mathcal{H}$, or $g$ superordinates $f$,  written
$f(z)\prec g(z)$ $(z\in \UD )$, if there exists a Schwarz function
$w$, analytic in $\UD $ with $w(0)=0$ and $|w(z)|<1$,  satisfying
$f(z)=g(w(z))$.  If the function $g$ is univalent in $\UD $, then
$f(z) \prec g(z)$ is equivalent to $f(0)=g(0)$ and $f(\UD)\subseteq
g(\UD)$. An exposition on the widely used theory of differential
subordination, developed in the main by Miller and Mocanu, with
numerous applications to univalent functions can be found in their
monograph \cite{miller}. Miller and Mocanu \cite{Miller2003} also
introduced the dual concept of differential superordination. Let $p,
h \in \mathcal{ H}$ and $\phi(r,s,t;z):\mathbb{ C}^3\times \UD
\rightarrow \mathbb{ C}$. If $p$ and $\phi(p(z),zp'(z),z^2p''(z);z)$
are univalent and $p$ satisfies the second-order superordination
\begin{equation}\label{super}
h(z)\prec \phi(p(z),zp'(z),z^2p''(z);z),\end{equation} then $p$ is a
solution of the differential superordination (\ref{super}). An
analytic function $q$ is called a \emph{ subordinant} if $q\prec p$
for all $p$ satisfying (\ref{super}). A univalent subordinant
$\widetilde{q}$ satisfying $q\prec \widetilde{q}$ for all
subordinants $q$ of (\ref{super}) is said to be the best
subordinant. Miller and Mocanu \cite{Miller2003} obtained conditions
on $h$, $q$ and $\phi$ for which the following differential
implication holds:
\[ h(z)\prec \phi(p(z),zp'(z),z^2p''(z);z) \Rightarrow q(z)\prec p(z).\]
Using these results, Bulboac\u a gave a treatment on certain classes
of first-order differential superordinations \cite{Bul1,Bul2002b},
as well as superordination-preserving integral operators
\cite{Bul2002}. Ali {\em et al.} \cite{ali} gave several
applications of first-order differential subordination and
superordination to obtain sufficient conditions for normalized
analytic functions $f$ to satisfy $q_1(z)\prec zf'(z)/f(z)\prec
q_2(z)$,  where $q_1$ and $q_2$ are given univalent analytic
functions in $\UD $. In \cite{ali2} they have also applied
differential superordination to functions defined by means of linear
operators. Recently Ali and Ravichandran \cite{ali3} investigated
first-order superordination to a class of meromorphic
$\alpha$-convex functions. Several differential subordination and
superordination associated with various linear operators were also
investigated in \cite{ros5}.

Generalizing the familiar starlike  and convex functions,
Lewandowski {\em et al.}\ \cite{lewan1} introduced $\gamma$-starlike
functions consisting of $f\in \mathcal{A}$ satisfying the inequality
\[  \RE\left( \left( \frac{zf'(z)}{f(z)} \right)^{1-\gamma}
 \left( 1+\frac{zf''(z)}{f'(z)} \right)^{\gamma}\right)>0. \]
These functions are starlike.  With $p(z):=zf'(z)/f(z)$, to show
that $\gamma$-starlike functions are indeed starlike is to
analytically  make the implication
\[ \RE\left(  p(z)\left(1+\frac{zp'(z)}{p^2(z)}
\right)^\gamma \right) >0 \Rightarrow \RE p(z)>0. \]
 Following the work of
Lewandowski {\em et al.}\ \cite{lewan1,lewan2}, Kanas {\em et al.}\
\cite{kanas1} determined conditions on $p$ and $h$ satisfying \[
p(z) \left( 1+\frac{zp'(z)}{p(z)}\right)^\alpha\prec h(z)
\Rightarrow p(z)\prec h(z)\] for a fixed $ \alpha\in [0,1]$.  Lecko
\cite{lecko2} (see Kanas {\em et al.}  \cite{kanas1} for a symmetric
version) investigated the more general subordination
\[ p(z)\left( 1+\frac{zp'(z)}{p(z)}\varphi(p(z))\right)^\alpha\prec h(z)
\Rightarrow p(z)\prec h(z).\] Singh and Gupta \cite{susma}
subsequently investigated the following first-order differential
subordination that included the important Briot-Bouquet differential
subordination:
\[ (p(z))^\alpha \left( p(z) + \frac{zp'(z)}{\beta
p(z)+\gamma} \right)^\mu \prec  (q(z))^\alpha \left( q(z) +
\frac{zq'(z)}{\beta q(z)+\gamma} \right)^\mu  \Rightarrow p(z)\prec
q(z). \] For a closely  related  class, see   S. Kanas\ and\ J.
Kowalczyk \cite{kanas0}.

The present paper investigates differential subordination and
superordination implications of expressions similar to the form
considered above by Singh and Gupta \cite{susma}. Special cases of
the results obtained include one involving the expresssion $\alpha
p^2(z) +(1-\alpha)p(z) + \alpha z p'(z)$, a result which cannot be
deduced from the work of Singh and Gupta \cite{susma}. The
sandwich-type results obtained in our present investigation are then
applied to normalized analytic univalent functions and to
$\Phi$-like functions.

The following definition and  results will be required:

\begin{lemma}[cf. Miller and Mocanu \protect{\cite[Theorem 3.4h, p.132]{miller}}]
\label{lem1} Let $q$ be univalent in the unit disk $\UD$, and let
$\vartheta$ and $\varphi$ be analytic in a domain $D \supset q(\UD)$
with $\varphi(w)\neq 0,$ $w \in q(\UD).$ With $
Q(z):=zq'(z)\varphi(q(z))$, let $ h(z):=\vartheta(q(z))+Q(z)$.
Suppose that  $Q$ is starlike univalent in $\UD $ and
\[{\RE\;}\left(\frac{zh'(z)}{Q(z)}\right)>0 \quad (z\in \UD).\]
 If $p$ is
analytic in $\UD $ with $p(0)=q(0)$, $p(\UD )\subset D$ and
\begin{equation*} \label{eq8}
\vartheta(p(z))+zp'(z)\varphi(p(z))\prec
\vartheta(q(z))+zq'(z)\varphi(q(z)),
\end{equation*} then $p(z)\prec q(z),$ and
$q$ is the best dominant.
\end{lemma}

\begin{definition}[\protect{\cite[Definition 2, p. 817]{Miller2003}}]
Denote by $\mathcal{Q} $ the set of all functions $f$ that are
analytic and injective on $\overline{\UD }-E(f)$, where
\[ E(f)=\{\zeta \in\partial \UD : \lim_{z\rightarrow \zeta}
f(z)=\infty \},\]  and are such that $f'(\zeta)\not=0$ for
$\zeta\in\partial \UD -E(f)$.
\end{definition}

\begin{lemma}[\protect{\cite{Bul2002b}}]
\label{lem2}
 Let $q$ be univalent in the unit disk
$\UD $, $\vartheta$ and $\varphi$ be analytic in a domain $D$
containing $q(\UD )$. Suppose that
$\RE\left[\vartheta'(q(z))/\varphi(q(z))\right]> 0
 \text{ for } z\in \UD $ and
 $zq'(z)\varphi(q(z))$ is starlike univalent in $\UD $.
If $p \in \mathcal{H}[q(0),1]\cap \mathcal{Q}$ with $p(\UD
)\subseteq D$, and $\vartheta(p(z))+zp'(z)\varphi(p(z))$ is
univalent in $\UD $, then
\begin{equation*}
 \label{lem2s1}
\vartheta(q(z))+zq'(z)\varphi(q(z)) \prec
\vartheta(p(z))+zp'(z)\varphi(p(z))
\end{equation*}
implies  $q(z)\prec p(z),$  and $q$ is the best subordinant.
\end{lemma}

\section{A Sandwich Theorem}

Our main result involves the following class of functions:

\begin{definition}
Let $\alpha$ and $\mu$ be fixed numbers with $0<\mu\leq 1$,
$\alpha+\mu\geq 0$. Also let $\beta$, $\gamma$ and $\delta$ be
complex numbers with $\beta\not=0$. The class $\mathcal{R}(\alpha,
\beta, \gamma,\delta,\mu)$ consists of analytic functions $p$ with
$p(0)=1$, $p(z)\not=0$ in $\UD,$ and are such that the functions
\[
P(z):=  (p(z))^\alpha \left( p(z)+\delta + \frac{zp'(z)}{\beta
p(z)+\gamma} \right)^\mu\quad (z\in \UD )
\]
are well-defined in $\UD $. (Here the powers are principal values.)
\end{definition}

By making use of Lemma~\ref{lem1}, the following result is derived:
\begin{theorem} \label{th1}
Let $q\in \mathcal{R}(\alpha, \beta, \gamma,\delta,\mu)$ be analytic
and univalent in $\UD $. Set
\begin{equation} \label{th2e1}
 R(z):= \frac{zq'(z)}{\beta q(z)+\gamma} \quad (z\in \UD ).
 \end{equation}
Assume that
\begin{equation} \label{th2e2}
 \RE\left( (\beta q(z)+\gamma)\left(1+\frac{\alpha}{\mu}
+\frac{\alpha\delta}{\mu q(z)}\right)  \right) > 0 \quad (z\in \UD
), \end{equation}
 and
\begin{equation} \label{th2e3} \RE\left(\frac{\alpha}{\mu}\frac{zq'(z)}{q(z)}
+\frac{zR'(z)}{R(z)}  \right) > 0 \quad (z\in \UD ).
\end{equation} If $p\in \mathcal{R}(\alpha, \beta,
\gamma,\delta,\mu)$ satisfies
\begin{equation}\label{sub1}
(p(z))^\alpha \left( p(z)+\delta + \frac{zp'(z)}{\beta p(z)+\gamma}
\right)^\mu \prec  (q(z))^\alpha \left( q(z)+\delta +
\frac{zq'(z)}{\beta q(z)+\gamma} \right)^\mu,
\end{equation}
then $p(z)\prec q(z),$ and $q$ is the best dominant.
\end{theorem}
\begin{proof}
We first write the differential subordination (\ref{sub1}) as
\[ (p(z))^{\frac{\alpha}{\mu}+1}  + \delta (p(z))^{\frac{\alpha}{\mu}} +
(p(z))^{\frac{\alpha}{\mu}}  \frac{zp'(z)}{\beta p(z)+\gamma} \prec
(q(z))^{\frac{\alpha}{\mu}+1}  + \delta (q(z))^{\frac{\alpha}{\mu}}
+ (q(z))^{\frac{\alpha}{\mu}} \frac{zq'(z)}{\beta q(z)+\gamma}.  \]
Define the functions $\vartheta$ and $\varphi$ by
\[\vartheta(w):= w^{\frac{\alpha}{\mu}+1} +\delta w^{\frac{\alpha}{\mu}} \text{ and }
\varphi(w):=\frac{w^{\frac{\alpha}{\mu}}}{\beta w+\gamma}.   \]
Since $q\in \mathcal{R}(\alpha, \beta, \gamma,\delta,\mu)$, then
$q(z)\not=0$ and therefore $\varphi(w)\not=0$ when $w\in q(\UD )$.
Also $\varphi$ and $\vartheta$ are analytic in a domain containing
$q(\UD )$. Define the function
\[ Q(z):= zq'(z)\varphi(q(z))=
(q(z))^{\frac{\alpha}{\mu}} \frac{zq'(z)}{\beta q(z)+\gamma} =
(q(z))^{\frac{\alpha}{\mu}} R(z), \] where $R$ is given by
(\ref{th2e1}). It follows from (\ref{th2e3}) that
\[ \RE \frac{zQ'(z)}{Q(z)} = \Re\left( \frac{\alpha}{\mu}\frac{zq'(z)}{q(z)}
+\frac{zR'(z)}{R(z)}  \right) >0,  \] and so $Q$ is a starlike
function. Now define $h$ by
\[ h(z):= \vartheta(q(z)) + Q(z)= (q(z))^{\frac{\alpha}{\mu}+1}  + \delta
(q(z))^{\frac{\alpha}{\mu}}+Q(z).\] In view of  the assumptions
(\ref{th2e2}) and (\ref{th2e3}), it follows that
\[ \RE\frac{zh'(z)}{Q(z)}
= \RE\left\{ (\beta q(z)+\gamma)\left(1+\frac{\alpha}{\mu}
+\frac{\alpha\delta}{\mu q(z)}\right)
+\frac{\alpha}{\mu}\frac{zq'(z)}{q(z)} +\frac{zR'(z)}{R(z)} \right\}
> 0 \quad (z\in \UD ).  \] The result is now deduced from
Lemma~\ref{lem1}.
\end{proof}

\begin{example} Let $q:\mathbb{D}\rightarrow \mathbb{C}$  be defined by
$q(z)=(1+Az)/(1+Bz)$ with $-1 < B < A\leq 1$. It is evident that
$q\in \mathcal{R}(\alpha, \beta, \gamma,\delta,\mu)$ whenever \[
\delta+\frac{1-A}{1-B} > \frac{A-B}{(1-B) |\,|\beta+\gamma|-|\beta
A+\gamma B)|\,|}.\] With additional constraints on the parameters,
there exists functions $q$ satisfying the hypothesis of
Theorem~\ref{th1}. For instance, in addition to the above condition,
assuming  all the parameters $\alpha, \beta, \gamma,\delta,$ and
$\mu$ are positive with \[ \frac{1-2A}{1-A} > \frac{|\beta A +\gamma
B|}{|\beta+\gamma-|\beta A +\gamma B|},
\]  then $q$ satisfies the conditions of Theorem~\ref{th1}.
\end{example}

By a similar  application of Lemma~\ref{lem2}, the following result
can be established, which we state without proof.

\begin{theorem}\label{th2}
Let $q\in \mathcal{R}(\alpha, \beta, \gamma,\delta,\mu)$ be  as in
Theorem~\ref{th1}.  Let $p\in \mathcal{R}(\alpha, \beta,
\gamma,\delta,\mu)$ satisfies $p\in \mathcal{H}\cap \mathcal{Q}$ and
$(p(z))^{\frac{\alpha}{\mu}+1}  + \delta (p(z))^{\frac{\alpha}{\mu}}
+ (p(z))^{\frac{\alpha}{\mu}} \frac{zp'(z)}{\beta p(z)+\gamma}$ be
univalent. If $p$ satisfies
\begin{equation*}\label{sub2}
(q(z))^\alpha \left( q(z)+\delta + \frac{zq'(z)}{\beta q(z)+\gamma}
\right)^\mu \prec  (p(z))^\alpha \left( p(z)+\delta +
\frac{zp'(z)}{\beta p(z)+\gamma} \right)^\mu,
\end{equation*}
then $q(z)\prec p(z),$ and $q$ is the best subordinant.
\end{theorem}

Combining Theorems~\ref{th1} and \ref{th2}, the following ``sandwich
theorem'' is obtained:

\begin{theorem}\label{th3}
Let $q_i\in \mathcal{R}(\alpha, \beta, \gamma,\delta,\mu)$ $(i=1,2)$
be analytic and univalent in $\UD $. Set
\[ R_i(z):=  \frac{zq_i'(z)}{\beta q_i(z)+\gamma} \quad (i=1,2; z\in \UD ),\]
\[ h_i(z):= (q_i(z))^\alpha \left( q_i(z)+\delta + \frac{zq_i'(z)}{\beta
q_i(z)+\gamma} \right)^\mu \quad (i=1,2).\] Assume that
\[ \RE\left( (\beta q_i(z)+\gamma)\left(1+\frac{\alpha}{\mu}
+\frac{\alpha\delta}{\mu q_i(z)} \right) \right) > 0 \quad (z\in \UD
) \] and
\[ \RE\left(\frac{\alpha}{\mu}\frac{zq_i'(z)}{q_i(z)}
+\frac{zR_i'(z)}{R_i(z)}  \right) >0 \quad (i=1,2; z\in \UD ).
\] If $p\in \mathcal{R}(\alpha, \beta, \gamma,\delta,\mu)$ satisfies $p\in
\mathcal{H}\cap \mathcal{Q}$ and $(p(z))^{\frac{\alpha}{\mu}+1}  +
\delta (p(z))^{\frac{\alpha}{\mu}} + (p(z))^{\frac{\alpha}{\mu}}
\frac{zp'(z)}{\beta p(z)+\gamma}$ is univalent, then
\begin{equation}\label{sub3}
h_1(z)  \prec  (p(z))^\alpha \left( p(z)+\delta +
\frac{zp'(z)}{\beta p(z)+\gamma} \right)^\mu \prec h_2(z)
\end{equation} implies $q_1(z)\prec p(z)
\prec q_2(z).$ Further $q_1$ and $q_2$ are respectively the  best
subordinant and the best dominant.
\end{theorem}

\section{Applications to Univalent Functions}

By use of Theorem~\ref{th3}, the following result is obtained:
\begin{theorem}\label{th4}
Let $\alpha$, $\mu$ be fixed numbers with $0<\mu\leq 1$, $\alpha+\mu
> 0$, and $\lambda \in\mathbb{C}$. Let $f, g\in\mathcal{A}$, and
$q_i(z)=zg_i'(z)/g_i(z)$ $(i=1,2)$ be univalent in $\UD $ satisfying
\[ \RE\left( \frac{1}{\lambda} q_i(z) \right) > 0 \]
and
\[  \RE\left( \left( \frac{\alpha}{\mu} -1\right) \frac{zq_i'(z)}{q_i(z)}
+ 1+ \frac{zq_i''(z)}{q_i'(z)}  \right) > 0 .  \] Let \[ h_i(z) :=
\left(\frac{zg_i'(z)}{g_i(z)}\right)^\alpha \left(
(1-\lambda)\frac{zg_i'(z)}{g_i(z)} + \lambda
\left(1+\frac{zg_i''(z)}{g_i'(z)}\right)\right)^\mu \quad (i=1,2).
\] If $f\in\mathcal{A}$ satisfies $0 \neq \frac{zf'(z)}{f(z)}\in
\mathcal{H}[1,1]\cap \mathcal{Q}$ and
$\left(\frac{zf'(z)}{f(z)}\right)^\alpha \left(
(1-\lambda)\frac{zf'(z)}{f(z)} + \lambda
\left(1+\frac{zf''(z)}{f'(z)}\right)   \right)^\mu$ is univalent in
$\UD $, then
\[
h_1(z) \prec \left(\frac{zf'(z)}{f(z)}\right)^\alpha \left(
(1-\lambda)\frac{zf'(z)}{f(z)} + \lambda
\left(1+\frac{zf''(z)}{f'(z)}\right)   \right)^\mu \prec h_2(z)\]
implies
\[
\frac{zg_1'(z)}{g_1(z)}\prec  \frac{zf'(z)}{f(z)}\prec
\frac{zg_2'(z)}{g_2(z)}.
 \]
\end{theorem}

\begin{proof}
The result follows from Theorem~\ref{th3} by taking
$\gamma=\delta=0$, $\beta=1/\lambda$, and
\[ p(z):=  \frac{zf'(z)}{f(z)} \text{ and } q_i(z) :=
\frac{zg_i'(z)}{g_i(z)}\quad (i=1,2). \qedhere\]
\end{proof}

The following two  corollaries are immediate consequences of
Theorem~\ref{th3} (or Theorem~\ref{th4}):
\begin{corollary} \cite{ali}
Let $\alpha \in\mathbb{C}$, and $q_i(z)\not=0$ ($i=1,2$) be
univalent in $\UD $. Assume that $ \RE\left[ \overline{\alpha}
q_i(z) \right]> 0 $ for $i=1,2$ and $zq_i'(z)/q_i(z)$ ($i=1,2$) is
starlike univalent in $\UD $. If $f\in \mathcal{A}$,
$0\not=zf'(z)/f(z)\in \mathcal{H}[1,1]\cap \mathcal{Q}$,
$(1-\alpha)\frac{zf'(z)}{f(z)}+\alpha\left(1+\frac{zf''(z)}{f'(z)}\right)$
is univalent in $\UD $, then
\[ q_1(z)+ \alpha \frac{zq_1'(z)}{q_1(z)}\prec
 (1-\alpha)\frac{zf'(z)}{f(z)}+\alpha\left(1+\frac{zf''(z)}{f'(z)}\right)
 \prec q_2(z)+ \alpha \frac{zq_2'(z)}{q_2(z)}\]
implies
\[  q_1(z)\prec \frac{zf'(z)}{f(z)}\prec q_2(z).\]
 Further $q_1$ and $q_2$ are respectively the best subordinant and best dominant.
\end{corollary}

\begin{corollary}  \cite{ali}
 Let $q_i(z)\not=0$ be  univalent in $\UD $ with $ \RE q_i(z)> 0$. Let
$zq_i'(z)/q_i^2(z)$ be starlike univalent in $\UD $ for $i=1,2$. If
$f\in\mathcal{A}$, $0\not=zf'(z)/f(z)\in \mathcal{H}[1,1]\cap Q$,
$\frac{1+zf''(z)/f'(z)}{zf'(z)/f(z)} $ is univalent in $\UD $, then
\[ 1+ \frac{zq_1'(z)}{q_1^2(z)}
\prec \frac{1+zf''(z)/f'(z)}{zf'(z)/f(z)} \prec 1+
\frac{zq_2'(z)}{q_2^2(z)}\] implies $q_1(z)\prec zf'(z)/f(z)\prec
q_2(z)$. Further $q_1$ and $q_2$ are respectively the best
subordinant and best dominant.
\end{corollary}

Another application of Theorem~\ref{th3} yields the following
result:
\begin{corollary} \label{corali} \cite{ali}
Let  $q_1$ and $q_2$ be convex univalent  in $\UD $. Let $0 \neq
\alpha \in\mathbb{C}$, and assume that $ \RE q_i(z)>
\RE\frac{\alpha-1}{2\alpha}$ for $i=1,2$. If $f\in \mathcal{A}$,
$zf'(z)/f(z)\in \mathcal{H}[1,1]\cap \mathcal{Q}$,
$\frac{zf'(z)}{f(z)}+\alpha\frac{z^2f''(z)}{f(z)}$ is univalent in
$\UD $, then
\[ (1-\alpha)q_1(z)+\alpha q_1^2(z)+\alpha zq_1'(z) \prec
\frac{zf'(z)}{f(z)}\left(1+\alpha\frac{zf''(z)}{f'(z)} \right) \prec
(1-\alpha)q_2(z)+\alpha q_2^2(z)+\alpha zq_2'(z)\] implies
\[ q_1(z)\prec \frac{zf'(z)}{f(z)}\prec q_2(z).\]
Further $q_1$ and $q_2$ are respectively the best subordinant and
best dominant.
\end{corollary}

\section{Application to $\Phi$-like  Functions}

Let $\Phi$ be an analytic function  in a domain containing $f(\UD
)$, $\Phi(0) = 0$ and  $\Phi'(0) >  0$. A function $f \in
\mathcal{A}$ is called {\em $\Phi$-like } if
\[ \RE\frac{zf^{'}(z)}{\Phi(f(z))} >0  \quad (z\in\UD ). \]
This concept was introduced by Brickman \cite{brick} and it was
shown that  an analytic function $f\in \mathcal{A}$ is univalent if
and only if $f$ is $\Phi$-like for some $\Phi$. When $\Phi(w)=w$ and
$\Phi(w)=\lambda w$, the  $\Phi$-like function $f$ is respectively
starlike and spirallike of type $\arg \lambda$. Ruscheweyh
\cite{rus} introduced and studied  the following general class of
$\Phi$-like functions:
\begin{definition}
Let $\Phi$ be analytic in a domain containing $f(\UD )$, $\Phi(0) =
0$, $\Phi'(0) = 1$ and $\Phi(\omega)\neq 0$ for $\omega \in f(\UD )
- \{ 0 \}.$ Let $q$ be a fixed analytic function in $\UD $, with
$q(0)=1$. A function $f \in \mathcal{A}$ is called $\Phi$-like with
respect to $q$ if
\[ \frac{zf^{'}(z)}{\Phi(f(z))} \prec q(z) \quad (z\in\UD ). \]
\end{definition}

\begin{theorem}\label{phith1} Let $\alpha\not= 0$ be a complex number
and  $q_i$ $(i=1,2)$ be convex univalent in $\UD $. Define $h_i$ by
\begin{equation*}
h_i(z) := \alpha q_i^2(z) +(1-\alpha)q_i(z) + \alpha z q_i'(z) \quad
(i=1,2),
\end{equation*}
and suppose that
\begin{equation*}\label{phieq1}
\RE \left( \frac{1-\alpha}{\alpha} + 2q_i(z)  \right ) >  0 \quad
(i=1,2; z\in\UD ).
\end{equation*}
If $f\in \mathcal{A}$ satisfies $f\in\mathcal{H}[1,1]\cap
\mathcal{Q}$ and $\frac{zf'(z)}{\Phi (f(z))} \left(  1 +
\frac{\alpha zf''(z)}{f'(z)} + \frac{\alpha z (f'(z) - (\Phi
(f(z)))')}{\Phi (f(z))}\right )$ is univalent in $\UD $, then
\begin{equation}\label{phieq2}
h_1(z) \prec \frac{zf'(z)}{\Phi (f(z))} \left (  1 + \frac{\alpha
zf''(z)}{f'(z)}
 + \frac{\alpha z (f'(z) - (\Phi (f(z)))')}{\Phi (f(z))}\right
) \prec  h_2(z) \end{equation} implies
\[
q_1(z) \prec \frac{zf'(z)}{\Phi(f(z))} \prec q_2(z).
\]
Further $q_1$ and $q_2$ are respectively the best subordinant and
the  best dominant.
\end{theorem}
\begin{proof}
Define the function $p$ by
\begin{equation}\label{phieq3}
p(z) := \frac{zf'(z)}{\Phi(f(z))} \quad (z\in\UD ).
\end{equation}
Then the function $p$ is analytic in $\UD $ with $p(0)=1$. From
(\ref{phieq3}), it follows that
\begin{eqnarray}\label{phieq5} &&
\frac{zf'(z)}{\Phi (f(z))} \left (  1 + \frac{\alpha zf''(z)}{f'(z)}
+ \frac{\alpha z (f'(z) - (\Phi (f(z)))')}{\Phi (f(z))}\right
)\nonumber \\ &=& p(z)\left(1 + \alpha \left ( \frac{zp'(z)}{p(z)} -
1 \right ) + \alpha p(z) \right ) \nonumber
\\ & = & \alpha p^2(z) +(1-\alpha)p(z) + \alpha z p'(z).
\end{eqnarray}
Putting (\ref{phieq5}) in the subordination (\ref{phieq2}) yields
\begin{equation*}\label{phieq6}
h_1(z) \prec \alpha p^2(z) +(1-\alpha)p(z) + \alpha z p'(z) \prec
h_2(z).
\end{equation*}
The result now follows from Theorem~\ref{th3}.
\end{proof}

\begin{remark}When $\Phi(w)=w$, Theorem \ref{phith1} reduces to
Corollary \ref{corali}.
\end{remark}


\begin{thebibliography}{00}

\bibitem{ali3} R. M. Ali, V. Ravichandran, Classes of meromorphic
$\alpha$-convex functions, \emph{Taiwanese J.  Math.,}\  \textbf{14}
(2010), no. 4,  1479--1490.

\bibitem{ali} R. M. Ali, V. Ravichandran, M. Hussain Khan and
K. G. Subramanian, Differential sandwich theorems for certain
analytic functions, {\em Far East J. Math.\ Sci.},\
  15  (2004), no. 1, 87--94.

\bibitem{ali2} R. M. Ali, V. Ravichandran, M. Hussain Khan and
K. G. Subramanian, Applications of first order differential
superordinations to  certain linear operators, \emph{Southeast Asian
Bull.\ Math.}\  30  (2006), no. 5, 799--810.


\bibitem{ros5}  R. M. Ali, V. Ravichandran, N. Seenivasagan,
Differential subordination and superordination of analytic functions
defined by the Dziok-Srivastava linear operator, \emph{J. Franklin
Inst.},\ \textbf{347} (2010), 1762--1781.



\bibitem{brick} L.\ Brickman, $\Phi$-like analytic functions,  I,
{\em Bull.\ Amer.\ Math.\ Soc.}, {\bf 79} (1973), 555--558.

\bibitem{Bul1} T. Bulboac\u a, Sandwich-type theorems for a class of
integral operators, \emph{Bull. Belg. Math. Soc. Simon Stevin} {\bf
13} (2006), no.~3, 537--550.

\bibitem{Bul2002} T. Bulboac\u a, A class of
superordination-preserving integral operators, \emph{Indag. Math.
(N.S.)} {\bf 13} (2002), no.~3, 301--311.

\bibitem{Bul2002b}T. Bulboac\u a, Classes of first-order
differential superordinations,\emph{ Demonstratio Math.} {\bf 35}
(2002), no.~2, 287--292.

\bibitem{kanas0} S. Kanas\ and\ J. Kowalczyk, A note on
Briot-Bouquet-Bernoulli differential subordination, Comment. Math.
Univ. Carolin. {\bf 46} (2005), no.~2, 339--347.

\bibitem{kanas1}  S. Kanas, A. Lecko\ and\ J. Stankiewicz,
Differential subordinations and geometric means, \emph{Complex
Variables Theory Appl.} {\bf 28} (1996), no.~3, 201--209.





\bibitem{lecko2} A. Lecko, On differential subordinations and
inclusion relation between classes of analytic functions,
\emph{Complex Variables Theory Appl.} {\bf 40} (2000), no.~4,
371--385.

\bibitem{lewan1}Z. Lewandowski, S. Miller\ and\ E. Z\l otkiewicz,
Gamma-starlike functions,\emph{ Ann. Univ. Mariae Curie-Sk\l odowska
Sect. A} {\bf 28} (1974), 53--58 (1976).

\bibitem{lewan2}  Z. Lewandowski, S. Miller\ and\ E. Z\l otkiewicz,
Generating functions for some classes of univalent functions,
\emph{Proc. Amer. Math. Soc.} {\bf 56} (1976), 111--117.



\bibitem{miller} S. S. Miller and P. T. Mocanu,
Differential Subordinations: Theory and Applications, {\em Series in
Pure and Applied Mathematics, No. 225}, Marcel Dekker, New York,
(2000).

\bibitem{Miller2003}  S. S. Miller and P. T. Mocanu, Subordinants of
differential superordinations, {\em Complex Variables}, {\bf 48}
(2003), no. 10,  815--826.


\bibitem{rus}  St.\ Ruscheweyh,   A  subordination theorem for
$\Phi-$like functions, {\em J.\ London Math.\ Soc.} (2) {\bf 13}
(1976), 275--280.



\bibitem{susma} S. Singh and S. Gupta, Some applications of a
first order differential subordination, {\em J. Inequal. Pure Appl.
Math.} {\bf 5}(2004), no.~3, Article 78.



\end{thebibliography}
\end{document}